\documentclass{amsart}
\usepackage{amsrefs}

\newtheorem{Theorem}{Theorem}[section]
\newtheorem{Lemma}[Theorem]{Lemma}

\numberwithin{equation}{section}

\newcommand{\R}{\mathbb{R}}
\newcommand{\cR}{\mathcal{R}}
\renewcommand{\S}{\mathbb{S}}
\newcommand{\cS}{\mathcal{S}}
\newcommand{\del}{\partial}
\newcommand{\grad}{\nabla}
\renewcommand{\epsilon}{\varepsilon}
\renewcommand{\[}{\begin{equation}}
\renewcommand{\]}{\end{equation}}
\newcommand{\Lip}{\text{Lip}}
\newcommand{\Area}{\operatorname{Area}}
\newcommand{\cat}{\text{cat}}

\begin{document}

\title{Maximal spectral surfaces of revolution converge to a catenoid}
\author{Sinan Ariturk}
\address{Pontif\'icia Universidade Cat\'olica do Rio de Janeiro \\ Rio de Janeiro, Brazil}
\email{ariturk@mat.puc-rio.br}

\begin{abstract}
We consider a maximization problem for eigenvalues of the Laplace-Beltrami operator on surfaces of revolution in $\R^3$ with two prescribed boundary components.
For every $j$, we show there is a surface $\Sigma_j$ which maximizes the $j$-th Dirichlet eigenvalue.
The maximizing surface has a meridian which is a rectifiable curve.
If there is a catenoid which is the unique area minimizing surface with the prescribed boundary, then the eigenvalue maximizing surfaces of revolution converge to this catenoid.
\end{abstract}

\maketitle

\section{Introduction}

On a smoothly bounded planar domain $\Omega$, the Dirichlet eigenvalues form a sequence
\[
	0 < \lambda_1(\Omega) \le \lambda_2(\Omega) \le \lambda_3(\Omega) \le \ldots
\]
The relationship between the eigenvalues and the domain $\Omega$ is complicated.
An interesting problem is to find domains which minimize eigenvalues subject to a geometric constraint.
For example, Bucur, Buttazzo, and Henrot~\cite{BBH} showed that there is a domain which minimizes the second eigenvalue among bounded planar domains with the same perimeter.
This was extended to higher eigenvalues by van den Berg and Iversen~\cite{BI}.
Bucur and Freitas~\cite{BF} showed that these domains converge to a disc.
These domains have been analyzed numerically by Antunes and Freitas~\cite{AF1} and Bogosel and Oudet~\cite{BO}.
A minimization problem with more general constraints was considered by van den Berg~\cite{Be}.

In this article, we consider a related problem for surfaces in $\R^3$ motivated by minimal surfaces and the Plateau problem.
Recall that a minimal surface is a surface that locally minimizes area or, equivalently, a surface with zero mean curvature.
The Plateau problem asks for an area minimizing surface with prescribed boundary.
For example, if the prescribed boundary is a planar circle in $\R^3$, then a planar disc is the unique area minimizing surface.

Fix two distinct parallel circles $C_1$ and $C_2$ in $\R^3$ centered about a common axis $\Gamma$.
Assume the circles are non-degenerate, i.e. have positive radius, and let $R_1$ and $R_2$ be the radii of $C_1$ and $C_2$, respectively.
There is an area minimizing surface with boundary given by $C_1$ and $C_2$.
If the circles are coplanar, then a planar annulus is the unique area minimizing surface.
If the circles are not coplanar, then there are three cases, depending on the choice of the circles.
In the first case, the union of two planar discs is the unique area minimizer.
In the second case, there is a catenoid which is the unique area minimizer. 
In the third case, there are two area minimizers.
One is the union of two planar discs and the other is a catenoid.

We pose an eigenvalue optimization problem in a similar spirit.
Let $\cS$ be the set of connected surfaces of revolution $\Sigma$ satisfying the following three properties.
First $\Sigma$ has two boundary components, given by the circles $C_1$ and $C_2$.
Second $\Sigma$ is disjoint from the axis $\Gamma$.
Third a meridian of $\Sigma$ is an oriented rectifiable curve from $C_1$ to $C_2$.
Recall a meridian of $\Sigma$ is the intersection of $\Sigma$ and a half-plane with boundary given by the axis $\Gamma$.
A rectifiable curve is a curve which admits a Lipschitz continuous parametrization.
For a smooth surface of revolution $\Sigma$ in $\cS$, let $\Delta_\Sigma$ be the Laplace-Beltrami operator on $\Sigma$.
The Dirichlet eigenvalues of $-\Delta_\Sigma$ form an increasing sequence,
\[
\label{evseq}
	0 < \lambda_1(\Sigma) < \lambda_2(\Sigma) \le \lambda_3(\Sigma) \le \ldots
\]
In the next section, we extend the domain of these eigenvalues to all surfaces in $\cS$, see \eqref{lamjdef}.
Then we study the maximization of these eigenvalues.
For every positive integer $j$, define
\[
	\Lambda_j = \sup \bigg\{ \lambda_j(\Sigma) : \Sigma \in \cS \bigg\}
\]

If $C_1$ and $C_2$ are coplanar, then the planar annulus $A$ in $\cS$ maximizes every eigenvalue, i.e. $\lambda_j(A)=\Lambda_j$ for all $j$.
This follows from~\cite[Theorem 1.1]{annulus} and an approximation argument, see Lemma \ref{annuluscomp} below.
If the circles $C_1$ and $C_2$ are not coplanar, then the following theorem establishes the existence of eigenvalue maximizing surfaces.
Let $D_1$ and $D_2$ be the planar discs embedded in $\R^3$ with boundaries given by $C_1$ and $C_2$, respectively.
For each positive integer $j$, let $\lambda_j(D_1 \cup ~D_2)$ be the $j$-th Dirichlet eigenvalue of $D_1 \cup D_2$.

\begin{Theorem}
\label{exist}
Fix a positive integer $j$, and assume that $\Lambda_j > \lambda_j(D_1 \cup D_2)$.
Then there is a surface $\Sigma_j$ in $\cS$ such that $\lambda_j (\Sigma_j) = \Lambda_j$.
\end{Theorem}

The case $j=1$ of Theorem \ref{exist} was established in~\cite{first}.
Moreover, in this case the maximizing surface was shown to be smooth.
We prove Theorem \ref{exist} in the third section, using an argument based on the Arzela-Ascoli theorem.
In order to show that a maximizing sequence is contained in a compact set, the key step is to show that a maximizing sequence of surfaces is uniformly bounded between two cylinders about the axis $\Gamma$.
This follows from an annulus comparison inequality, see Lemma \ref{dblannuluscomp}.
This inequality was established in~\cite{annulus} for piecewise smooth surfaces of revolution, and we extend it to surfaces in $\cS$ by an approximation argument.

If there is a catenoid which is the unique area minimizing surface with boundary given by $C_1$ and $C_2$, then the following theorem establishes that the eigenvalue maximizing surfaces given by Theorem \ref{exist} converge to this catenoid.

\begin{Theorem}
\label{catenoid}
Assume there is a catenoid $\Sigma_\cat$ which is the unique area minimizing surface with boundary given by $C_1$ and $C_2$.
Then $\Lambda_j > \lambda_j(D_1 \cup D_2)$ for large $j$.
For such $j$, let $\Sigma_j$ be a surface in $\cS$ such that $\lambda_j (\Sigma_j) = \Lambda_j$.
Then
\[
\label{arealim}
	\lim_{j \to \infty} \Area(\Sigma_j) = \Area(\Sigma_\cat)
\]
Moreover $\Sigma_j$ converges to $\Sigma_\cat$ in the Hausdorff metric.
\end{Theorem}

We prove Theorem \ref{catenoid} in the fourth section, using an argument based on Weyl's law.
If $\Sigma$ is a surface in $\cS$, then Weyl's law states that
\[
\label{weyl}
	\lim_{j \to \infty} \frac{\lambda_j(\Sigma)}{j} = \frac{4 \pi}{\Area(\Sigma)}
\]
Note that Weyl's law easily yields the inequality $\Lambda_j > \lambda_j(D_1 \cup D_2)$ for large $j$ in Theorem \ref{catenoid}.
The argument we use to prove Theorem \ref{catenoid} yields a stronger result, 
see Lemma \ref{strongcat}.
Namely we can weaken the hypothesis that $\lambda_j(\Sigma_j)=\Lambda_j$ for large $j$ and instead only assume that
\[
\label{sketch1}
	\liminf_{j \to \infty} \frac{\lambda_j(\Sigma_j)}{j} \ge \frac{4\pi}{\Area(\Sigma_\cat)}
\]
Note that this is a weaker assumption by Weyl's law \eqref{weyl}.
To prove \eqref{arealim}, we show that
\[
\label{sketch2}
	\liminf_{j \to \infty} \frac{\lambda_j(\Sigma_j)}{j} \le \liminf_{j \to \infty} \frac{4\pi}{\Area(\Sigma_j)}
\]
Note that \eqref{sketch1} and \eqref{sketch2} imply \eqref{arealim}, because $\Sigma_\cat$ is area minimizing.
The proof of \eqref{sketch2} relies crucially on the assumption that $\Sigma_\cat$ is the unique area minimizing surface.
This assumption enables us to prove that there are two cylinders of positive radius centered about the axis $\Gamma$ such that each surface $\Sigma_j$ is in the region bounded between these cylinders.
Then we can prove that the surfaces $\Sigma_j$ have meridians with uniformly bounded length.
For any sequence of surfaces in $\cS$ satisfying these two properties, we can prove \eqref{sketch2}.
Finally we use \eqref{arealim} to prove that the surfaces $\Sigma_j$ converge to $\Sigma_\cat$ in the Hausdorff metric.
We are not able to draw any conclusions when $D_1 \cup D_2$ is an area minimizer.
Specifically, we are not able to prove that the surfaces $\Sigma_j$ are uniformly bounded between cylinders about the axis $\Gamma$ or that their meridians have uniformly bounded length.
Therefore, our proof of \eqref{sketch2} does not apply, and we are not able to establish a statement analogous to \eqref{arealim} or any convergence of the eigenvalue maximizing surfaces.

As described earlier, these results are closely related to the problem of minimizing Dirichlet eigenvalues among planar domains with fixed perimeter.
Another related problem is to minimize Dirichlet eigenvalues among Euclidean domains of fixed volume.
The Faber-Krahn inequality states that a ball minimizes the first Dirichlet eigenvalue among such sets.
By the Krahn-Szeg\"o inequality, the union of two balls of equal radii minimizes the second Dirichlet eigenvalue.
Bucur and Henrot~\cite{BH} showed that there is a quasi-open set which minimizes the third eigenvalue.
This was extended to higher eigenvalues by Bucur~\cite{Bu}.
A similar result for any increasing functional of finitely many eigenvalues was established by Mazzoleni and Pratelli~\cite{MP}.

Although eigenvalue optimizing domains are often not known explicitly, there are other situations where the limit of eigenvalue optimizing domains is an explicit shape.
Antunes and Freitas~\cite{AF2} showed that the rectangle of given area which minimizes the $j$-th Dirichlet eigenvalue converges to a square.
A similar result for the rectangle which maximizes the $j$-th Neumann eigenvalue was established by van den Berg, Bucur, and Gittins~\cite{BBG}.
Antunes and Freitas~\cite{AF1} showed that the rectangular parallelepiped which minimizes the $j$-th Dirichlet eigenvalue subject to a surface area constraint on the boundary converges to a cube.
The asymptotic behavior of eigenvalue minimizing domains is related to the P\'olya conjecture.
In two dimensions, P\'olya's conjecture states that if $\Omega$ is a planar domain with area one, then
\[
\label{polyaconj}
	\lambda_j(\Omega) \ge 4 \pi j
\]
If $\Omega_j^*$ minimizes the $j$-th Dirichlet eigenvalue among planar domains with volume one, then Colbois and El Soufi~\cite{CE} showed that \eqref{polyaconj} is equivalent to
\[
	\lim_{j \to \infty} \frac{\lambda_j(\Omega_j^*)}{j} = 4 \pi
\]
We remark that P\'olya's conjecture and the equivalence theorem of Colbois and El Soufi extend to higher dimensions.

This article is particularly inspired by work of Abreu and Freitas~\cite{AbF}, who studied optimization of eigenvalues of $\S^1$-invariant metrics on $\S^2$.
They bounded the $\S^1$-invariant eigenvalues of a surface of revolution in $\R^3$ in terms of the eigenvalues of a disc.
Colbois, Dryden, and El Soufi~\cite{CDE} extended this to higher dimensions.
These results were motivated by Hersch~\cite{H}, who showed that on a sphere, the round metric maximizes the first non-zero eigenvalue among metrics of given area.

In the next section we define the eigenvalues $\lambda_j(\Sigma)$ for any surface $\Sigma$ in $\cS$, and we reformulate the problem in terms of curves in a half-plane.
In the third section we prove Theorem \ref{exist}, and in the fourth section we prove Theorem \ref{catenoid}.

\section{Eigenvalues on Rectifiable Surfaces of Revolution}

In this section we define the eigenvalues $\lambda_j(\Sigma)$ for any surface $\Sigma$ in $\cS$ and we discuss some basic facts.
Recall that a surface in $\cS$ has a meridian which is a rectifiable curve, i.e. admits a Lipschitz continuous parametrization.
For a smooth surface $\Sigma$ in $\cS$, the definition of the eigenvalues $\lambda_j(\Sigma)$ is well known.
The main purpose of this section is to extend this definition to surfaces in $\cS$ with low regularity.

Fix an open half-plane in $\R^3$ whose boundary is given by the axis $\Gamma$.
Identify this half-plane with $\R^2_+$, defined by
\[
	\R^2_+ = \bigg\{ (x,y) \in \R^2 : x>0 \bigg\}
\]
Let $p$ and $q$ be the points where $\R^2_+$ intersects the circles $C_1$ and $C_2$, respectively.
Let $\cR$ be the set of Lipschitz functions $\alpha:[0,1] \to \R^2_+$ satisfying the following two properties.
First $\alpha(0)=p$ and $\alpha(1)=q$.
Second there is a constant $L$ such that $|\alpha'(t)|=L$ for almost every $t$.
Note that there is a bijective correspondence between $\cR$ and $\cS$. 

Before defining the eigenvalues for surfaces in $\cS$ with low regularity, we recall a well known variational characterization of the eigenvalues on a smooth surface.
If $\Sigma$ is a smooth surface in $\cS$, then let $\alpha$ be the corresponding smooth curve in $\cR$.
Define $\lambda_j(\alpha) = \lambda_j(\Sigma)$ for each $j$.
Let $\Lip_0(\Sigma)$ be the space of functions $f:\Sigma \to \R$ which are Lipschitz and vanish on the boundary $\del \Sigma$.
Then for each $j$,
\[
	\lambda_j(\alpha) = \min_V \max_{f \in V} \frac{\int_\Sigma | \grad f |^2 \,dS}{\int_\Sigma |f|^2 \,dS}
\]
Here the minimum is taken over all $j$-dimensional subspaces $V$ of $\Lip_0(\Sigma)$.
Also $\grad$ is the Riemannian gradient and $dS$ is the Riemannian measure on $\Sigma$.
Separation of variables shows that every eigenvalue can be realized by an eigenfunction which is the product of a radially symmetric function and a sinusoidal function.
Therefore we may express the eigenvalues in an alternative way.
Write $\alpha=(F,G)$, i.e. let $F$ and $G$ be the component functions of $\alpha$.
Let $\Lip_0(0,1)$ be the space of functions $w:[0,1] \to \R$ which are Lipschitz and vanish at the endpoints $0$ and $1$.
For a non-negative integer $k$ and a positive integer $n$, define
\[
\label{lamkndef}
	\lambda_{k,n}(\alpha) = \min_W \max_{w \in W} \frac{\int_0^1 \frac{|w'|^2 F}{|\alpha'|} + \frac{k^2 |w|^2 |\alpha'|}{F} \,dt}{\int_0^1 |w|^2 F |\alpha'| \,dt}
\]
Here the minimum is taken over all $n$-dimensional subspaces $W$ of $\Lip_0(0,1)$.
Then
\[
\label{lamknlamj}
	\Big\{ \lambda_{k,n}(\alpha) \Big\} = \Big\{ \lambda_j(\alpha) \Big\}
\]
Moreover, counting the eigenvalue $\lambda_{k,n}(\alpha)$ twice for $k \neq 0$, the multiplicities agree.

Note that the right side of \eqref{lamkndef} makes sense if the smooth curve $\alpha$ is replaced with a Lipschitz continuous curve in $\cR$.
For any curve $\alpha$ in $\cR$, we define $\lambda_{k,n}(\alpha)$ for non-negative $k$ and positive $n$ by \eqref{lamkndef}.
It is well known that if $\alpha$ is smooth, then there are only finitely many eigenvalues $\lambda_{k,n}(\alpha)$ in any bounded subset of $\R$.
This also holds when $\alpha$ is a Lipschitz continuous curve in $\cR$.
This can be shown by comparing the eigenvalues of $\alpha$ to the eigenvalues of a cylinder.
Hence, if $\alpha$ is a Lipschitz continuous curve in $\cR$, we may define the eigenvalues $\lambda_j(\alpha)$ for positive integers $j$ as follows.
Counting the eigenvalue $\lambda_{k,n}(\alpha)$ twice for $k \neq 0$, we define $\lambda_j(\alpha)$ for positive integers $j$ so that \eqref{evseq} and \eqref{lamknlamj} hold, with multiplicities agreeing in \eqref{lamknlamj}.
Let $\Sigma$ be the low regularity surface in $\cS$ corresponding to $\alpha$, and define $\lambda_{k,n}(\Sigma)=\lambda_{k,n}(\alpha)$ for all $k$ and $n$.
Likewise, for all $j$, define
\[
\label{lamjdef}
	\lambda_j(\Sigma)=\lambda_j(\alpha)
\]
We also define the area of $\Sigma$ by
\[
	\Area(\Sigma) = 2\pi \int_0^1 F |\alpha'| \,dt
\]

These definitions of eigenvalues also make sense for a slightly more general class of curves.
Let $[c,d]$ be an interval and let $\alpha:[c,d] \to \R^2_+$ be Lipschitz.
If there is a constant $\delta>0$ such that $|\alpha'(t)| \ge \delta$ for almost every $t$ in $[c,d]$, then we may define the eigenvalues $\lambda_{k,n}(\alpha)$ and $\lambda_j(\alpha)$ similarly.
For $k$ fixed, the eigenvalues $\lambda_{k,n}(\alpha)$ are the eigenvalues of the regular Sturm-Liouville problem
\[
	- \bigg( \frac{F w'}{|\alpha'|} \bigg)' + \frac{k^2 |\alpha'| w}{F} = \lambda F |\alpha'| w
\]
with boundary conditions $w(c)=w(d)=0$.
This observation yields a continuity result for the functionals $\lambda_{k,n}$.
To state this, assume there is a constant $L>0$ such that $|\alpha'(t)|=L$ for almost every $t$ in $[c,d]$.
For positive integers $m$, let $\alpha_m:[c,d] \to \R^2_+$ be a Lipschitz curve, and write $\alpha_m=(F_m,G_m)$.
Assume $F_m$ converges to $F$ uniformly over $[c,d]$ and assume $|\alpha_m'|$ converges in $L^\infty(c,d)$ to the constant function $L$.
For every $k$ and $n$,
\[
\label{kzcont}
	\lim_{m \to \infty} \lambda_{k,n}(\alpha_m) = \lambda_{k,n}(\alpha)
\]
This holds by a theorem on continuity of eigenvalues of regular Sturm-Liouville problems~\cite[Theorem 3.1]{KZ}.

\section{Existence of Maximizers}
In this section we prove Theorem \ref{exist}.
In Lemma \ref{annuluscomp}, we prove that if a surface in $\cS$ is projected onto a plane, then the eigenvalues of the resulting annulus bound the eigenvalues of the original surface.
Then in Lemma \ref{dblannuluscomp}, we show that if a surface in $\cS$ has a large eigenvalue, then it must be bounded between two cylinders about the axis $\Gamma$.
In Lemma \ref{lengthbound}, we bound the length of a surface in $\cS$ with a large eigenvalue.
Then we prove Theorem \ref{exist}, using the Arzela-Ascoli theorem.

\begin{Lemma}
\label{annuluscomp}
Let $\alpha:[c,d] \to \R^2_+$ be Lipschitz.
Assume there is a constant $L$ such that $|\alpha'(t)|=L$ for almost every $t$ in $[c,d]$.
Write $\alpha=(F_\alpha,G_\alpha)$.
Let $[\rho_1, \rho_2]$ be the image of $F_\alpha$, and assume $\rho_1<\rho_2$.
Let $A$ be an annulus in $\R^2$ of radii $\rho_1$ and $\rho_2$.
Then $\lambda_{k,n}(\alpha) \le \lambda_{k,n}(A)$ for every $k$ and $n$.
\end{Lemma}

\begin{proof}
There is a subinterval of $[c,d]$ such that $\alpha$ attains its minimum and maximum values at the endpoints of the subinterval.
By the domain monotonicity of Dirichlet eigenvalues, the eigenvalues of the restriction of $\alpha$ to this subinterval are greater than or equal to the eigenvalues of $\alpha$.
Passing to the restriction, it suffices to consider the case where $F_\alpha$ attains its minimum and maximum values $\rho_1$ and $\rho_2$ at the endpoints $c$ and $d$.
Furthermore, by a change of variables, we may assume that the domain $[c,d]$ is equal to $[0,1]$.
Note that if $\alpha$ is smooth and regular, then $\lambda_{k,n}(\alpha) \le \lambda_{k,n}(A)$ for every $k$ and $n$ by~\cite[Lemma 2.1]{annulus}.
First we use an approximation argument to extend these inequalities to continuously differentiable curves.
Then we use another approximation argument to extend these inequalities to Lipschitz curves.

For now assume that the curve $\alpha:[0,1] \to \R^2_+$ is continously differentiable and there is a constant $L$ such that $|\alpha'(t)|=L$ for every $t$ in $[c,d]$.
Define a function $\gamma:[-1,2] \to \R^2_+$ by
\[
\label{extendalpha}
	\gamma(t) =
	\begin{cases}
		\alpha(t)=2 \alpha(0)-\alpha(-t) & t \in [-1,0] \\
		\alpha(t) & t \in [0,1] \\
		\alpha(t)=2 \alpha(1)-\alpha(2-t) & t \in [1,2] \\
	\end{cases}
\]
Then the curve $\gamma$ is continuously differentiable over $[-1,2]$, with $|\gamma'(t)| = L$ for every $t$ in $[-1,2]$.
Write $\gamma=(F_\gamma,G_\gamma)$.
For $\epsilon>0$, let $F_\epsilon$ and $G_\epsilon$ be the standard mollifications of $F_\gamma$ and $G_\gamma$, respectively, see e.g.~\cite[pp. 122-123]{EG}.
Restrict the domains of $F_\epsilon$ and $G_\epsilon$ to $[0,1]$.
Then $F_\epsilon$ and $G_\epsilon$ are smooth.
Moreover $F_\epsilon$ and $G_\epsilon$ converge uniformly in $C^1$ to $F_\alpha$ and $G_\alpha$, respectively.
Define $\gamma_\epsilon=(F_\epsilon, G_\epsilon)$.
Then by \eqref{kzcont},
\[
	\lim_{\epsilon \to 0} \lambda_{k,n}(\gamma_\epsilon) = \lambda_{k,n}(\alpha)
\]
Since the standard mollifier is even and decreasing over $[0,\infty)$, the extenstion \eqref{extendalpha} yields $\rho_1 \le F_\epsilon(t) \le \rho_2$ for all $\epsilon>0$ and all $t$ in $[0,1]$.
Furthermore $F_\epsilon(0)=\rho_1$ and $F_\epsilon(1)=\rho_2$.
Note that $\lambda_{k,n}(\gamma_\epsilon) \le \lambda_{k,n}(A)$ for every $\epsilon$, $k$ and $n$ by the result mentioned above, because $\gamma_\epsilon$ is a smooth regular curve and the image of $F_\epsilon$ is $[\rho_1,\rho_2]$.
Therefore $\lambda_{k,n}(\alpha) \le \lambda_{k,n}(A)$ for every $k$ and $n$.
This completes the proof for the case where $\alpha$ is continuously differentiable.

We complete the proof by using another approximation argument.
Now assume that the curve $\alpha:[0,1] \to \R^2_+$ is Lipschitz and there is a constant $L$ such that $|\alpha'(t)|=L$ for almost every $t$ in $[0,1]$.
The argument differs from the previous one in the definition of $G_\epsilon$.
Define $\gamma:[-1,2] \to \R^2_+$ by \eqref{extendalpha}.
Write $\gamma=(F_\gamma,G_\gamma)$.
For $\epsilon>0$, let $F_\epsilon$ be the standard mollification of $F_\gamma$.
Restrict the domain of $F_\epsilon$ to $[0,1]$.
Then $F_\epsilon$ is smooth and $|F_\epsilon'(t)| \le L$ for every $t$ in $[0,1]$.
Also the image of $F_\epsilon$ is $[\rho_1,\rho_2]$.
Define $G_\epsilon: [0,1] \to \R$ so that $G_\epsilon(0)=G_\alpha(0)$ and for all $t$ in $[0,1]$,
\[
	G_\epsilon'(t) = \sqrt{L^2 - (F_\epsilon'(t))^2}
\]
Define $\gamma_\epsilon=(F_\epsilon, G_\epsilon)$.
Then $\gamma_\epsilon$ is continuously differentiable and $|\gamma_\epsilon'(t)|=L$ for all $t$ in $[0,1]$.
Moreover $F_\epsilon$ converges to $F_\alpha$ uniformly over $[0,1]$ as $\epsilon \to 0$.
By \eqref{kzcont},
\[
	\lim_{\epsilon \to 0^+} \lambda_{k,n}(\gamma_\epsilon) = \lambda_{k,n}(\alpha)
\]
Note that $\lambda_{k,n}(\gamma_\epsilon) \le \lambda_{k,n}(A)$ for every $\epsilon$, $k$ and $n$, by the argument above, because $\gamma_\epsilon$ is a continuously differentiable curve such that $|\gamma_\epsilon'(t)|=L$ for every $t$ in $[0,1]$, and the image of $F_\epsilon$ is $[\rho_1, \rho_2]$.
Therefore $\lambda_{k,n}(\alpha) \le \lambda_{k,n}(A)$  for every $k$ and $n$.
This completes the proof.
\end{proof}

We use Lemma \ref{annuluscomp} to show that if a surface in $\cS$ has a large eigenvalue, then it must be bounded between two cylinders about the axis $\Gamma$.
Recall that $R_1$ and $R_2$ are the radii of $C_1$ and $C_2$, respectively.

\begin{Lemma}
\label{dblannuluscomp}
Fix constants $b>a>0$.
Assume $a$ is less than $R_1$ and $R_2$.
Also assume $b$ is greater than $R_1$ and $R_2$.
Let $A_1$ and $A_2$ be disjoint annuli in $\R^2$ with inner radii $a$ and outer radii $R_1$ and $R_2$, respectively.
Let $B_1$ and $B_2$ be disjoint annuli in $\R^2$ with outer radii $b$ and inner radii $R_1$ and $R_2$, respectively.
Let $\alpha$ be a curve in $\cR$, and write $\alpha=(F,G)$.
If $\lambda_j(\alpha) > \lambda_j(A_1 \cup A_2)$ for some $j$, then $F(t) \ge a$ for all $t$ in $[0,1]$.
Likewise, if $\lambda_j(\alpha) > \lambda_j(B_1 \cup B_2)$ for some $j$, then $F(t) \le b$ for all $t$ in $[0,1]$.
\end{Lemma}

\begin{proof}
Assume that $\lambda_j(\alpha) > \lambda_j(A_1 \cup A_2)$ for some $j$.
Fix $t_0$ in $[0,1]$ and suppose that $F(t_0) <a$.
Note that $t_0 \neq 0$ and $t_0 \neq 1$, because $\alpha$ is in $\cR$.
Let $\beta$ and $\gamma$ be the restrictions of $\alpha$ to $[0,t_0]$ and $[t_0,1]$, respectively.
Then $\lambda_{k,n}(\beta) \le \lambda_{k,n}(A_1)$ for all $k$ and $n$, by Lemma \ref{annuluscomp} and by the domain monotonicity of Dirichlet eigenvalues.
Hence $\lambda_j(\beta) \le \lambda_j(A_1)$ for all $j$.
Likewise $\lambda_j(\gamma) \le \lambda_j(A_2)$ for all $j$.
Therefore $\lambda_j(\alpha) \le \lambda_j(A_1 \cup A_2)$ for all $j$.
This contradiction shows that $F(t) \ge a$ for all $t$ in $[0,1]$.
A similar argument can be used to prove that if $\lambda_j(\alpha) > \lambda_j(B_1 \cup B_2)$ for some $j$, then $F(t) \le b$ for all $t$ in $[0,1]$.
\end{proof}

The following lemma provides an upper bound for eigenvalues of a curve in terms of the length.

\begin{Lemma}
\label{lengthbound}
Let $\alpha$ be a curve in $\cR$.
Write $\alpha=(F,G)$, and let $L$ be the length of $\alpha$.
Assume that there are positive constants $a$ and $b$ such that $a \le F(t) \le b$ for every $t$ in $[0,1]$.
Fix $j$.
Then
\[
	\lambda_j(\alpha) \le \frac{j^2 \pi^2 b}{L^2 a}
\]
\end{Lemma}

\begin{proof}
For each $i=1,2,\ldots,j$, define a function $w_i$ in $\Lip_0(0,1)$ by
\[
	w_i(t) =
	\begin{cases}
		\sin(j \pi t) & \frac{i-1}{j} \le t \le \frac{i}{j} \\
		0 & \text{otherwise} \\
	\end{cases}
\]
For each $i$,
\[
\label{lengthbound1}
	\frac{\int_0^1 \frac{|w_i'|^2 F}{|\alpha'|} \,dt}{\int_0^1 |w_i|^2 F |\alpha'| \,dt}
	\le \frac{j^2 \pi^2 b}{L^2 a} 
\]
Let $W$ be the $j$-dimensional subspace of $\Lip_0(0,1)$ generated by the functions $w_i$.
Since the supports of these functions are disjoint, it follows from \eqref{lengthbound1} that
\[
	\max_{w \in W} \frac{\int_0^1 \frac{|w'|^2 F}{|\alpha'|} \,dt}{\int_0^1 |w|^2 F |\alpha'| \,dt} \le \frac{j^2 \pi^2 b}{L^2 a}
\]
Therefore
\[
	\lambda_{0,j}(\alpha) \le \max_{w \in W} \frac{\int_0^1 \frac{|w'|^2 F}{|\alpha'|} \,dt}{\int_0^1 |w|^2 F |\alpha'| \,dt}
	\le \frac{j^2 \pi^2 b}{L^2 a}
\]
Since $\lambda_j(\alpha) \le \lambda_{0,j}(\alpha)$, this completes the proof.
\end{proof}

Now we can prove Theorem \ref{exist}.

\begin{proof}[Proof of Theorem 1.1]
Fix $j$.
Fix constants $b>a>0$.
Assume $a$ is less than $R_1$ and $R_2$.
Also assume $b$ is greater than $R_1$ and $R_2$.
Let $A_1$ and $A_2$ be disjoint annuli in $\R^2$ with inner radii $a$ and outer radii $R_1$ and $R_2$, respectively.
Let $B_1$ and $B_2$ be disjoint annuli in $\R^2$ with outer radii $b$ and inner radii $R_1$ and $R_2$, respectively.
Recall that $\Lambda_j > \lambda_j(D_1 \cup D_2)$, by assumption.
Therefore, if $a$ is small, then $\Lambda_j > \lambda_j(A_1 \cup A_2)$.
For a proof of this, we refer to Rauch and Taylor~\cite{RT}, who considered a much more general problem.
If $b$ is large, then $\Lambda_j > \lambda_j(B_1 \cup B_2)$.
Let $\Sigma_i$ be a sequence in $\cS$ such that
\[
	\lim_{i \to \infty} \lambda_j(\Sigma_i) = \Lambda_j
\]
We may assume that $\lambda_j(\Sigma_i) > \lambda_j(A_1 \cup A_2)$ and $\lambda_j(\Sigma_i) > \lambda_j(B_1 \cup B_2)$ for every $i$.
Let $\alpha_i$ be the curve in $\cR$ corresponding to $\Sigma_i$.
Write $\alpha_i=(F_i,G_i)$.
Then $a \le F_i(t) \le b$ for every $t$ in $[0,1]$ and every $i$, by Lemma \ref{dblannuluscomp}.
Let $L_i$ be the length of $\alpha_i$.
By Lemma~\ref{lengthbound}, the lengths $L_i$ are uniformly bounded.
By passing to a subsequence, we may assume that there is a $L>0$ such that the lengths $L_i$ converges to $L$.
By applying the Arzela-Ascoli theorem and passing to a subsequence, we may assume that the curves $\alpha_i$ converge uniformly to a curve $\alpha:[0,1] \to \R^2_+$.
Moreover $|\alpha'(t)| \le L$ for almost every $t$ in $[0,1]$.
Write $\alpha=(F,G)$.
Then $F:[0,1] \to [a,b]$ satisfies $|F'(t)| \le L$ for almost every $t$ in $[0,1]$.
There is a Lipschitz function $H:[0,1] \to \R$ such that, for almost every $t$ in $[0,1]$,
\[
	(F'(t))^2+(H'(t))^2 = L^2
\]
Note that $|H'(t)| \ge |G'(t)|$ for almost every $t$ in $[0,1]$.
Moreover we may choose the function $H$ so that $H(0)=G(0)$ and $H(1)=G(1)$.
Define $\beta=(F,H)$.
Then $\beta$ is in $\cR$.
By \eqref{kzcont}, for every $k$ and $n$,
\[
	\lambda_{k,n}(\beta) = \lim_{i \to \infty} \lambda_{k,n}(\alpha_i)
\]
Therefore
\[
	\lambda_j(\beta) \ge \limsup_{i \to \infty} \lambda_j(\alpha_i) = \Lambda_j
\]
Let $\Sigma_j$ be the surface in $\cS$ corresponding to $\beta$.
Then
\[
	\lambda_j(\Sigma_j) = \lambda_j(\beta) \ge \Lambda_j
\]
Since $\Sigma_j$ is in $\cS$, this implies that $\lambda_j(\Sigma_j) = \Lambda_j$.
\end{proof}

\section{Convergence to the Catenoid}

In this section, we prove Theorem \ref{catenoid}.
In fact we prove a slightly stronger statement, see Lemma~\ref{strongcat}.
In Lemma \ref{uniannulus}, we show that a sequence of surfaces with large eigenvalues is uniformly bounded between two cylinders about the axis $\Gamma$.
In Lemma \ref{unilength}, we show that such a sequence of surfaces have uniformly bounded length.
Lemma \ref{rect} provides bounds for eigenvalues on rectangles, which we later use to establish \eqref{sketch2}.
In Lemma \ref{hausdorff}, we show that surfaces with small area must approximate the minimizing catenoid.
Then we prove Lemma \ref{strongcat}, which establishes Theorem \ref{catenoid}.

\begin{Lemma}
\label{uniannulus}
Assume there is a catenoid $\Sigma_\cat$ which is the unique area minimizing surface with boundary given by $C_1$ and $C_2$.
Let $\alpha_j$ be curves in $\cR$ such that
\[
\label{jweylannulus}
	\liminf_{j \to \infty} \frac{\lambda_j(\alpha_j)}{j} \ge \frac{4\pi}{\Area(\Sigma_\cat)}
\]
Write $\alpha_j=(F_j,G_j)$.
Then there are positive constants $a$ and $b$ such that we have $a \le F_j(t) \le b$ for large $j$ and every $t$ in $[0,1]$.
\end{Lemma}

\begin{proof}
Fix constants $b>a>0$.
Assume $a$ is less than $R_1$ and $R_2$.
Also assume $b$ is greater than $R_1$ and $R_2$.
Let $A_1$ and $A_2$ be disjoint annuli in $\R^2$ of inner radii $a$ and outer radii $R_1$ and $R_2$, respectively.
Let $B_1$ and $B_2$ be disjoint annuli in $\R^2$ of outer radii $b$ and inner radii $R_1$ and $R_2$, respectively.
If $a$ is small and $b$ is large, then we have $\Area(\Sigma_\cat) < \Area(A_1 \cup A_2)$ and $\Area(\Sigma_\cat) < \Area(B_1 \cup B_2)$.
Then by Weyl's law,
\[
	\frac{4\pi}{\Area(\Sigma_\cat)} > \lim_{j \to \infty} \frac{\lambda_j(A_1 \cup A_2)}{j}
\]
Likewise,
\[
	\frac{4\pi}{\Area(\Sigma_\cat)} >  \lim_{j \to \infty} \frac{\lambda_j(B_1 \cup B_2)}{j}
\]
Hence \eqref{jweylannulus} implies that $\lambda_j(\alpha_j) >  \lambda_j(A_1 \cup A_2)$ and $\lambda_j(\alpha_j) >  \lambda_j(B_1 \cup B_2)$ for large $j$.
Therefore $a \le F_j(t) \le b$ for large $j$ and for every $t$ in $[0,1]$, by Lemma \ref{dblannuluscomp}.
\end{proof}

Next we show that the lengths of the eigenvalue maximizing curves are uniformly bounded.
We prove this by comparing the eigenvalues to those of a cylinder.

\begin{Lemma}
\label{unilength}
Assume there is a catenoid $\Sigma_\cat$ which is the unique area minimizing surface with boundary given by $C_1$ and $C_2$.
Let $\alpha_j$ be curves in $\cR$ such that
\[
\label{jweyllength}
	\liminf_{j \to \infty} \frac{\lambda_j(\alpha_j)}{j} \ge \frac{4\pi}{\Area(\Sigma_\cat)}
\]
Then the lengths $L_j$ of the curves $\alpha_j$ are uniformly bounded.
\end{Lemma}

\begin{proof}
Write $\alpha_j=(F_j,G_j)$.
There are positive constants $a$ and $b$ such that we have $a \le F_j(t) \le b$ for large $j$ and every $t$ in $[0,1]$, by Lemma \ref{uniannulus}.
Therefore, if $j$ is large, then for any function $w$ in $\Lip_0(0,1)$ and any $k$,
\[
\label{unilength1}
\begin{split}
	\frac{\int_0^1 \frac{|w'|^2 F_j}{|\alpha_j'|} + \frac{k^2 |w|^2 |\alpha_j'|}{F_j} \,dt}{\int_0^1 |w|^2 F_j |\alpha_j'| \,dt}
		&\le \frac{\int_0^1 \frac{|w'|^2 b}{L_j} + \frac{k^2 |w|^2 L_j}{a} \,dt}{\int_0^1 |w|^2 a L_j \,dt} \\
		&\le \frac{b}{a} \frac{\int_0^1 \frac{|w'|^2 a}{L_j} + \frac{k^2 |w|^2 L_j}{a} \,dt}{\int_0^1 |w|^2 a L_j \,dt} \\
\end{split}
\]
Let $C_j$ be a cylinder of radius $a$ and height $L_j$.
By \eqref{unilength1}, if $j$ is large, then for any $k$ and $n$,
\[
	\lambda_{k,n}(\alpha_j) \le \frac{b}{a} \lambda_{k,n}(C_j)
\]
Hence if $j$ is large,
\[
\label{unilength2}
	\lambda_j(\alpha_j) \le \frac{b}{a} \lambda_j(C_j)
\]
Suppose the lengths $L_j$ are unbounded.
Let $M>0$.
Then there are infinitely many $j$ such that $L_j \ge M$.
Let $C_M$ be a cylinder of radius $a$ and height $M$.
For all $j$ such that $L_j \ge M$, the domain monotonicity of Dirichlet eigenvalues implies that
\[
\label{unilength3}
	\lambda_j(C_j) \le \lambda_j(C_M)
\]
Therefore, by \eqref{unilength2}, \eqref{unilength3}, and Weyl's law,
\[
	\liminf_{j \to \infty} \frac{ \lambda_j(\alpha_j) }{j} \le \liminf_{j \to \infty} \frac{b}{a} \frac{ \lambda_j(C_j) }{j} \le \lim_{j \to \infty} \frac{b}{a} \frac{ \lambda_j(C_M) }{j} = \frac{b}{a} \frac{4\pi}{\Area(C_M)}
\]
Since $M$ may be arbitrarily large, this yields
\[
\label{unilength4}
	\liminf_{j \to \infty} \frac{ \lambda_j(\alpha_j) }{j} = 0
\]
Since \eqref{unilength4} and \eqref{jweyllength} are contradictory, the lengths $L_j$ must be uniformly bounded.
\end{proof}

The following lemma provides bounds for Dirichlet eigenvalues on a union of rectangles.
These estimates are useful because they only depend on the area and perimeter of the rectangles and are otherwise independent of the choice of rectangles.
These estimates can also be used to bound Dirichlet eigenvalues on cylinders.
In the proof of Theorem \ref{catenoid}, we apply these estimates to obtain bounds for the eigenvalues $\lambda_j(\Sigma_j)$ in terms of the area of $\Sigma_j$, in order to establish \eqref{sketch2}.

\begin{Lemma}
\label{rect}
Let $Q_1, \ldots Q_N$ be disjoint compact rectangles in $\R^2$.
Define $Q = \cup Q_i$.
For every $j$ such that $\lambda_j(Q) > 1$,
\[
\label{rect0}
	\frac{4\pi j}{\lambda_j(Q)-1} \ge \Area(Q) - 2 \operatorname{Perimeter}(Q) \Big(\lambda_j(Q)-1 \Big)^{-1/2}
\]
\end{Lemma}

\begin{proof}
For each $m=1,2,\ldots,N$, let $D(\lambda,Q_m)$ be the number of Dirichlet eigenvalues on $Q_m$ which are less than or equal to $\lambda$.
For every $\lambda>0$ and every $m$, it is well known that
\[
\label{rect1}
	\frac{4\pi D(\lambda,Q_m)}{\lambda} \ge \Area(Q_m) - 2 \operatorname{Perimeter}(Q_m) \lambda^{-1/2}
\]
See e.g.~\cite[p. 19]{L}.
Define $D(\lambda,Q)$ similarly, and note that
\[
	D(\lambda, Q) = D(\lambda,Q_1)+\ldots+D(\lambda,Q_N)
\]
Hence summing over $m$ in \eqref{rect1} yields
\[
	\frac{4\pi D(\lambda,Q)}{\lambda} \ge \Area(Q) - 2 \operatorname{Perimeter}(Q) \lambda^{-1/2}
\]
Fix $j$ such that $\lambda_j(Q) >1$ and set $\lambda=\lambda_j(Q)-1$.
Then $j \ge D(\lambda_j(Q) -1,Q)$, so \eqref{rect0} follows.
\end{proof}

If there is a catenoid which is the unique area minimizer, then the next lemma shows that surfaces in $\cS$ with small area must approximate the catenoid in the Hausdorff metric.

\begin{Lemma}
\label{hausdorff}
Assume there is a catenoid $\Sigma_\cat$ which is the unique area minimizing surface with boundary given by $C_1$ and $C_2$.
Let $\Sigma_j$ be a sequence of surfaces in $\cS$ such that
\[
\label{areasjcat}
	\lim_{j \to \infty} \Area(\Sigma_j) = \Area(\Sigma_\cat)
\]
Then the surfaces $\Sigma_j$ converge to $\Sigma_\cat$ in the Hausdorff metric.
\end{Lemma}

\begin{proof}
Fix constants $b>a>0$.
Assume $a$ is less than $R_1$ and $R_2$.
Also assume $b$ is greater than $R_1$ and $R_2$.
Let $A_1$ and $A_2$ be disjoint annuli in $\R^2$ with inner radii $a$ and outer radii $R_1$ and $R_2$, respectively.
Let $B_1$ and $B_2$ be disjoint annuli in $\R^2$ with outer radii $b$ and inner radii $R_1$ and $R_2$, respectively.
If $a$ is small and $b$ is large, then we have $\Area(\Sigma_\cat) < \Area(A_1 \cup A_2)$ and $\Area(\Sigma_\cat) < \Area(B_1 \cup B_2)$.
Then we may assume that for every $j$,
\[
\label{areaanncompa}
	\Area(\Sigma_j) < \Area(A_1 \cup A_2)
\]
and
\[
\label{areaanncompb}
	\Area(\Sigma_j) < \Area(B_1 \cup B_2)
\]

To prove convergence of the surfaces $\Sigma_j$, we show that any subsequence admits a subsequence which converges to $\Sigma_\cat$ in the Hausdorff metric.
Let $\Sigma_{j_k}$ be an arbitrary subsequence. 
For each $k$, let $\alpha_k$ be the curve in $\cR$ corresponding to $\Sigma_{j_k}$.
Write $\alpha_k=(F_k,G_k)$, and let $L_k$ be the length of $\alpha_k$.
Note that \eqref{areaanncompa} and \eqref{areaanncompb} imply that $a \le F_k(t) \le b$ for every $k$ and every $t$ in $[0,1]$.
It then follows that the lengths $L_k$ are uniformly bounded.
By passing to a subsequence, we may assume that the lengths $L_k$ converge to some positive constant $L$.
By the Arzela-Ascoli theorem, there is a subsequence $\alpha_{k_n}$ which converges uniformly to some Lipschitz curve $\beta:[0,1] \to \R^2_+$.
Moreover $|\beta'(t)| \le L$ for almost every $t$ in $[0,1]$.
Write $\beta=(F_\beta, G_\beta)$.
Now
\[
\label{betaarea}
	\int_0^1 F_\beta | \beta'| \le \int_0^1 F_\beta L = \lim_{n \to \infty} \int_0^1 F_{k_n} L_{k_n}
\]
Let $\Sigma_\beta$ be the surface in $\cS$ such that $\beta$ parametrizes a meridian of $\Sigma_\beta$.
Then by \eqref{areasjcat} and \eqref{betaarea},
\[
	\Area(\Sigma_\beta) \le \Area(\Sigma_\cat)
\]
Since $\Sigma_\cat$ is the unique area minimizer, this implies that $\Sigma_\beta=\Sigma_\cat$, i.e. $\beta$ parametrizes a meridian of $\Sigma_\cat$.
Now the uniform convergence of the curves $\alpha_{k_n}$ to $\beta$ implies that the surfaces in $\cS$ corresponding to $\alpha_{k_n}$ converge to $\Sigma_\cat$ in the Hausdorff metric.
That is, a subsequence of $\Sigma_{j_k}$ converges to $\Sigma_\cat$ in the Hausdorff metric.
Therefore the full sequence of surfaces $\Sigma_j$ converge to $\Sigma_\cat$ in the Hausdorff metric.
\end{proof}

Now we conclude the article by proving Theorem \ref{catenoid}.
Since $\Sigma_\cat$ is in $\cS$, Weyl's law \eqref{weyl} yields
\[
	\liminf_{j \to \infty} \frac{\Lambda_j}{j} \ge \lim_{j \to \infty} \frac{\lambda_j(\Sigma_\cat)}{j} = \frac{4\pi}{\Area(\Sigma_\cat)}
\]
Therefore Theorem \ref{catenoid} is a consequence of the following lemma.

\begin{Lemma}
\label{strongcat}
Assume there is a catenoid $\Sigma_\cat$ which is the unique area minimizing surface with boundary given by $C_1$ and $C_2$.
Let $\Sigma_j$ be a sequence of surfaces in $\cS$ such that
\[
\label{jweylconc}
	\liminf_{j \to \infty} \frac{\lambda_j(\Sigma_j)}{j} \ge \frac{4\pi}{\Area(\Sigma_\cat)}
\]
Then
\[
	\lim_{j \to \infty} \Area(\Sigma_j) = \Area(\Sigma_\cat)
\]
Moreover $\Sigma_j$ converges to $\Sigma_\cat$ in the Hausdorff metric.
\end{Lemma}

\begin{proof}
For each $j$, let $\alpha_j$ be the curve in $\cR$ corresponding to $\Sigma_j$, and let $L_j$ be the length of $\alpha_j$.
Let $0<\epsilon<1$.
Let $N>0$ be an integer, and partition $[0,1]$ into $N$ subintervals each of length $1/N$.
That is, for $m=1,2,\ldots,N$, define
\[
	I_m = [(m-1)/N, m/N]
\]
Let $\alpha_{j,m}$ be the restrictions of $\alpha_j$ to $I_m$ for each $m=1,2,\ldots,N$.
For each $j$ and $m$, write $\alpha_{j,m}=(F_{j,m}, G_{j,m})$.
Define $r_{j,m}$ to be the minimum of $F_{j,m}$.
Note that the maximum of $F_{j,m}$ is at most $r_{j,m}+L_j/N$.
Moreover by Lemma \ref{uniannulus} and Lemma~\ref{unilength}, the quantities $r_{j,m}$ and $L_j$ are uniformly bounded above and below by positive constants, independent of $j$ and $m$.
These bounds are also independent of $N$.
In particular, we may assume that $N$ is large, so that for every $j$ and $m$,
\[
\label{rjm}
	\frac{r_{j,m} + L_j/N}{r_{j,m}} \le 1+ \epsilon
\]

Therefore for every $j$, $k$ and $m$, and for every $w$ in $\Lip_0(I_m)$,
\[
\label{bjcj}
\begin{split}
	\frac{\int_{I_m} \frac{|w'|^2 F_{j,m}}{|\alpha_j'|} + \frac{k^2 |w|^2 |\alpha_j'|}{F_{j,m}} \,dt}{\int_{I_m} |w|^2 F_{j,m} |\alpha_j'| \,dt}
		&\le \frac{\int_{I_m} \frac{|w'|^2 (r_{j,m}+L_j/N)}{L_j} + \frac{k^2 |w|^2 L_j}{r_{j,m}} \,dt}{\int_{I_m} |w|^2 r_{j,m} L_j \,dt} \\
		&\le (1+\epsilon) \frac{\int_{I_m} \frac{|w'|^2 r_{j,m}}{L_j} + \frac{k^2 |w|^2 L_j}{r_{j,m}} \,dt}{\int_{I_m} |w|^2 r_{j,m} L_j \,dt} \\
\end{split}
\]
Let $C_{j,m}$ be a cylinder of radius $r_{j,m}$ and height $L_j/N$.
Then by \eqref{bjcj}, for all $j$, $k$, $m$, and $n$,
\[
\label{lambetajm}
	\lambda_{k,n}(\alpha_{j,m}) \le (1+\epsilon) \lambda_{k,n}(C_{j,m})
\]
For each $j$ and $m$, let $Q_{j,m}$ be a rectangle of width $2 \pi r_{j,m}$ and height $L_j/N$.
Assume the rectangles $Q_{j,m}$ are disjoint, and define
\[
	Q_j = \bigcup_{m=1}^N Q_{j,m}
\]
Then $\lambda_j(C_{j,m}) \le \lambda_j(Q_{j,m})$ for every $j$ and $m$, so \eqref{lambetajm} yields
\[
	\lambda_j(\Sigma_j) \le (1+\epsilon) \lambda_j(Q_j)
\]
By \eqref{rjm}, for every $j$,
\[
	\Area(\Sigma_j) \le (1+\epsilon) \Area(Q_j)
\]

Because of the uniform bounds on $r_{j,m}$ and $L_j$, there is a constant $C_N$ such that $\operatorname{Perimeter}(Q_j) \le C_N$ for every $j$.
Therefore if $j$ is large, then by Lemma \ref{rect},
\[
\label{bdperim}
	\frac{4 \pi j}{\lambda_j(\Sigma_j)-2} \ge \frac{\Area(\Sigma_j)}{(1+\epsilon)^2} - 2 C_N \Big(\lambda_j(\Sigma_j)-2 \Big)^{-1/2}
\]
Now by \eqref{jweylconc} and \eqref{bdperim},
\[
\begin{split}
	\limsup_{j \to \infty} \frac{\Area(\Sigma_j)}{(1+ \epsilon)^2}
		&= \limsup_{j \to \infty} \frac{\Area(\Sigma_j)}{(1+\epsilon)^2} - 2C_N \Big(\lambda_j(\Sigma_j)-2 \Big)^{-1/2} \\
		&\le \limsup_{j \to \infty} \frac{4\pi j}{\lambda_j(\Sigma_j)} \\
		&\le \Area(\Sigma_\cat) \\
\end{split}
\]
Since $0<\epsilon<1$ is arbitrary and $\Sigma_\cat$ is area minimizing, this yields
\[
	\lim_{j \to \infty} \Area(\Sigma_j) = \Area(\Sigma_\cat)
\]
Therefore by Lemma \ref{hausdorff}, the surfaces $\Sigma_j$ converge to $\Sigma_\cat$ in the Hausdorff metric.
\end{proof}

\end{document}